\documentclass[11pt]{amsart}
\usepackage{enumerate}
\usepackage{graphicx}
\usepackage{amssymb}

\newtheorem{theorem}{Theorem}[section]
\newtheorem{lemma}[theorem]{Lemma}
\newtheorem{definition}[theorem]{Definition}

\newtheorem{corollary}[theorem]{Corollary}
\newtheorem{remark}[theorem]{Remark}

\newcommand{\beqa}{\begin{eqnarray*}}
\newcommand{\eeqa}{\end{eqnarray*}}
\newcommand{\beqn}{\begin{eqnarray}}
\newcommand{\eeqn}{\end{eqnarray}}
\newcommand{\e}{\varepsilon}
\newcommand{\del}{\delta}
\newcommand{\ol}{\overline}

\newcommand{\C}{\mathcal{C}}
\newcommand{\K}{\mathcal{K}}
\newcommand{\W}{\mathcal{W}}

\newcommand{\ds}{\displaystyle}

\newcounter{cnt1}
\newcounter{cnt2}
\newcounter{cnt3}
\newcounter{cnt4}
\newcommand{\blr}{\begin{list}{$($\roman{cnt1}$)$}
{\usecounter{cnt1} \setlength{\topsep}{0pt}
\setlength{\itemsep}{0pt}}}
\newcommand{\bla}{\begin{list}{$($\alph{cnt2}$)$}
{\usecounter{cnt2} \setlength{\topsep}{0pt}
\setlength{\itemsep}{0pt}}}
\newcommand{\bln}{\begin{list}{$($\arabic{cnt3}$)$}
{\usecounter{cnt3} \setlength{\topsep}{0pt}
\setlength{\itemsep}{0pt}}}
\newcommand{\blR}{\begin{list}{$($\Roman{cnt4}$)$}
{\usecounter{cnt4} \setlength{\topsep}{0pt}
\setlength{\itemsep}{0pt}}}
\newcommand{\el}{\end{list}}

\begin{document}

\title[The Generalised MIP \& UMIP]{The Generalised (Uniform) Mazur Intersection Property}

\author[Bandyopadhyay]{Pradipta Bandyopadhyay}
\address[Pradipta Bandyopadhyay]{Stat--Math Division,
Indian Statistical Institute, 203, B.~T. Road, Kolkata
700108, India.}
\email{pradipta@isical.ac.in}

\author[Gothwal]{Deepak Gothwal}
\address[Deepak Gothwal]{Stat--Math Division, Indian Statistical Institute, 203, B.~T. Road, Kolkata
700108, India.}
\email{deepakgothwal190496@gmail.com}

\subjclass[2010]{Primary 46B20 \hfill To appear in JMAA}



\begin{abstract}
Given a family $\C$ of closed bounded convex sets in a Banach space $X$, we say that $X$ has the $\C$-MIP if every $C \in \C$ is the intersection of the closed balls containing it. In this paper, we introduce a stronger version of the $\C$-MIP and show that it is a more satisfactory generalisation of the MIP inasmuch as one can obtain complete analogues of various characterisations of the MIP.

We also introduce uniform versions of the (strong) $\C$-MIP and characterise them analogously. Even in this case, the strong $\C$-UMIP appears to have richer characterisations than the $\C$-UMIP.
\end{abstract}

\maketitle

\section{Introduction}
S.\ Mazur \cite{Ma} initiated the study of the following intersection property of balls in normed spaces, now called the Mazur Intersection Property (MIP):
\begin{quote}
{\it Every closed bounded convex set is the intersection of closed balls containing it.}
\end{quote}

J.~R.~Giles, D.\ A.\ Gregory and B.\ Sims~\cite{GGS} obtained various characterisations of the MIP, most well-known criterion stating that the w*-denting points of $B(X^*)$ are norm dense in $S(X^*)$. Chen and Lin \cite{CL} introduced the notion of w*-semidenting points to characterise the MIP. The paper \cite{GJM} is a good survey of the MIP and related works.

Whitfield and Zizler in \cite{WZ1} introduced a uniform version of the MIP (UMIP) and obtained characterisations similar to \cite{GGS} but the characterisation related to the w*-denting points was missing. Recently, such a characterisation in terms of ``uniform w*-semidenting points" has been obtained by the authors in \cite{BGG}.

Soon after \cite{GGS}, there appeared several papers dealing with generalisations of the MIP, basically studying when members of a given family of closed bounded convex sets are intersection of balls. For example, \bla
\item $\mathcal{K}$ = \{all compact convex sets in $X\}$ \cite{WZ, Se}
\item $\mathcal{W}$ = \{all weakly compact convex sets in $X\}$ \cite{Z}
\item $\mathcal{F}$ = \{all compact convex sets in $X$ with finite affine dimension$\}$ \cite{S2}
\el

Let $\C$ be a family of bounded subsets of $X$. Let $\tau_{\C}$ be the topology on $X^*$ of uniform convergence on the sets of $\C$. Let us say that $X$ has the $\C$-MIP if every closed convex set $C \in \C$ is the intersection of the closed balls containing it.

In \cite{Se, S2, WZ}, the authors characterised the $\C$-MIP in cases (a) and (c) above in terms of the density in $X^*$ of the cone of the extreme points of $B_{X^*}$ in the topology $\tau_{\C}$.

The first author studied the $\C$-MIP for a general family $\C$ of bounded sets in \cite{Ba2} and obtained some necessary and some sufficient conditions that recaptured the results of \cite{Se, S2, WZ}. Chen and Lin \cite{CL1} also discussed the $\C$-MIP. Both the papers observed that if the cone of the ``$\C$-denting points'' of $B(X^*)$, as defined in \cite{CL1}, is $\tau_\C$-dense in $X^*$, then $X$ has the $\C$-MIP. Both papers proved the converse under additional hypothesis.

In \cite{GGS}, there are several alternative characterisations of the MIP \cite[Theorem 2.1 $(ii), (iii)$ and $(iv)$]{GGS}. Analogous characterisations of the UMIP were also obtained in \cite{WZ, BGG}. \cite{Ba2} is the only attempt so far to obtain such analogues for the $\C$-MIP (See Theorem~\ref{T1} below).

Recently, Cheng and Dong \cite{CD} characterised the $\C$-MIP in terms of the ``$\C$-semidenting'' points. However, their definition does not appear to be a natural generalisation of either the w*-semidenting points of \cite{CL} or the $\C$-denting points of \cite{CL1}.

Vanderwerff \cite{Va} introduced the following stronger notion for the families $\mathcal{K}$ and $\mathcal{W}$ above, which we will call \emph{strong $\C$-MIP}:
\begin{quote}
\emph{For every convex $C \in \C$ and $\beta \geq 0$, $\ol{C + \beta B(X)}$ is the intersection of closed balls containing it.}
\end{quote}
For $\C = \K$, he showed that it is equivalent to the extreme points of $B(X^*)$ being w*-dense in $S(X^*)$.

In this paper, we show that the strong $\C$-MIP is actually a more satisfactory generalisation of the MIP for general families. It can be characterised in terms of a more natural notion of the (strong) $\C$-semidenting points. Moreover, one can also obtain complete analogue of \cite[Theorem 2.1]{GGS}.

We also introduce the uniform version of the $\C$-MIP and its stronger form and characterise them analogously. Even in this case, the strong $\C$-UMIP appears have richer characterisations than the $\C$-UMIP.

In the last section, we summarise the interrelations between various notions discussed in the paper with examples and counterexamples.

\section{Preliminaries}

Throughout this article, $X$ is a real Banach space.

For $x\in X$ and $r>0$, we denote by $B(x, r)$ the \emph{open ball} $\{y\in X : \|x-y\|<r\}$ and by $B[x, r]$ the \emph{closed ball} $\{y\in X: \|x-y\|\leq r\}$. We denote by $B(X)$ the \emph{closed unit ball} $\{x\in X : \|x\| \leq 1\}$ and by $S(X)$ the \emph{unit sphere} $\{x \in X : \|x\| = 1\}$.

For a bounded subset $A \subseteq X$, denote by co($A$) (resp.\ aco($A$)) the convex (resp.\ absolutely convex) hull of $A$. Let $\|f\|_A := \sup\{|f(x)| : x \in A\}$ for $f \in X^*$ and $diam_A(B) = \sup\{ \|f - g\|_A : f, g \in B\}$ for any non-empty subset $B$ of $X^*$.
For $f \in X^*$, $\e >0$ and $A \in \C$,
$B_{A}(f, \e) := \{g \in X^* : \|f - g\|_A <\e\}$.

For a bounded set $C \subseteq X$ and $x \in X$, let $d(x, C) = \inf\{\|x-z\| : z \in C\}$ denote the distance function.

For $A \subseteq X$, the cone generated by $A$ is cone$(A) = \{\lambda a : \lambda \geq 0, a \in A\}$.

\begin{definition} \rm
We say that a class $\C$ of bounded subsets of $X$ is a compatible class if \blR
\item $C \in \C$, $\alpha \in \mathbb{R}$ and $x \in X \implies \alpha C + x \in \C$, $C \cup \{x\} \in \C$;

\item $C \in \C$ implies the closed absolutely convex hull of $C \in \C$.

\item $C \in \C$ and $A \subseteq C$ implies $A \in \C$.

\item $C_1, C_2 \in \C$ implies $C_1 \cup C_2 \in \C$.
\el
\end{definition}

Our definition is slightly different from the ones in \cite{Ba2, CL1, CD}.

\begin{definition} \rm
For a compatible class $\C$, $X$ is said to have the $\C$-MIP if every \emph{closed convex} $C \in \C$ is the intersection of closed balls containing it.
\end{definition}

Notice that the families $\K$, $\W$ and $\mathcal{F}$ above are not exactly compatible classes with this definition. One has to consider such sets and subsets thereof. However, the property $\C$-MIP remains the same. And we should add the family $\mathcal{B}$ of all bounded sets, corresponding to the MIP itself, in our list of examples of a compatible class.

Throughout the article, $\C$ will denote a compatible class of sets. We will denote by $\tau_{\C}$, or, simply $\tau$ when there is no ambiguity, the topology on $X^*$ of uniform convergence on the sets of $\C$, or, in other words, the topology on $X^*$ generated by the family of seminorms $\{\|\cdot\|_C : C \in \C\}$. Since $\C$ is a compatible class, the family $\{B_{C}(0, \e) : C \in \C, \e > 0\}$ forms a local base for $\tau_{\C}$ at 0. For the classes $\C = \mathcal{F}$, $\K$, $\W$ and $\mathcal{B}$, the topology $\tau_{\C}$ coincides with the w*-, bw*-, Mackey and norm topologies, respectively. We note that the topologies $\tau_{\mathcal{F}}$ (w*-) and $\tau_{\K}$ (bw*-) coincide on bounded sets.

\begin{definition} \rm \bla
\item For $x\in S(X)$, we denote by $D(x)$ the set $\{f\in S(X^*) : f(x)=1\}$. Any selection of $D$ is called a support mapping.

\item A \emph{w*-slice} of $B(X^*)$ determined by $x \in S(X)$ is a set of the form
\[
S(B(X^*), x, \delta) := \{f \in B(X^*) : f(x) > 1-\delta \}
\]
for some $0<\delta < 1$.

\item For $\e, \delta>0$, $ x\in S(X) $ and $C \in \C$ denote by
\beqa
d_1(C, x, \delta) & = & \sup \limits_{0< \lambda < \delta, \ y \in C} \frac{\|{x+\lambda y}\| +\|{x-\lambda y}\|- 2}{\lambda} \\
d_2(C, x, \delta) & = & diam_{C}(S(B(X^*), x, \delta))\\
d_3(C, x, \delta) & = & diam_{C}(D(S(X)\cap B (x, \delta)))
\eeqa
\item \cite{Ba2} For $\e, \delta > 0$ and $C \in \C$
\[
M_{\e, \delta, C}(X) = \{x\in S(X) : \sup_{0< \|y\| < \delta, y \in C} \frac{\|{x + y}\| + \|{x - y}\|- 2}{\|y\|} < \e\}.
\]
In other words, $M_{\e, \delta, C}(X) = \{x\in S(X) : d_1(C, x, \delta )< \e\}$. Define
\[
M_{\e, C}(X) = \bigcup_{\delta > 0} M_{\e, \delta, C}(X).
\]

\item Define $H = \bigcap \{\ol{D(M_{\e, C}(X))}^{\tau} : C \in \C, \e >0\}$.
\el
\end{definition}

Following lemma is an easy adaptation of \cite[Lemma 2.10]{BGG}. Details can be found in \cite[Lemma 2.1]{Ba1}:

\begin{lemma} \label{lem1}
For any $\alpha, \delta > 0$, we have,
\begin{enumerate}[(i)]
\item $\ds d_2(C, x, \alpha )\leq d_1(C, x, \delta) +\frac{2\alpha}{\delta}$.
\item $d_3(C, x, \delta)\leq d_2(C, x, \delta )$.
\item $d_1(C, x, \delta)\leq d_3(C, x, 2{\delta})$.
\end{enumerate}
\end{lemma}

\begin{definition} \rm \bla
\item We say that $f \in S(X^*)$ is a w*-denting point of $B(X^*)$ if for every $\e > 0$, there exists a w*-slice $S$ of $B(X^*)$ such that $f \in S$ and $diam(S) < \e$.
\item \cite{CL} We say that $f \in S(X^*)$ is a w*-semidenting point of $B(X^*)$ if for every $\e > 0$, there exists a w*-slice $S$ of $B(X^*)$ such that $S \subseteq B(f, \e)$.
\item \cite{CL1} We say that $f \in S(X^*)$ is a $\C$-denting point of $B(X^*)$ if for each $A \in \C$ and $\e > 0$, there exists a w*-slice $S$ of $B(X^*)$ such that $f \in S$ and $diam_{A}(S) < \e$.
\item \cite{CD} We say that $f \in S(X^*)$ is a $\C$-semidenting point of $B(X^*)$ if for each $A \in \C$ and $\e > 0$, there exists a w*-slice $S$ of $B(X^*)$ such that $S \subseteq cone(B_{A}(f, \e))$.
\el
\end{definition}

Known results on $\C$-MIP \cite{Ba2, CL1, CD} are summarised in our notation and terminology below.

\begin{theorem} \cite[Theorem 1, Corollary 1]{Ba2}, \cite[Theorem 1.10]{CL1}, \cite[Theorem 2.3]{CD} \label{T1} Suppose $\C$ is a compatible class of bounded sets. Consider the following statements~:
\bla
\item The cone generated by $\C$-denting points of $B(X^*)$ is $\tau$-dense in~$X^*$.
\item The cone generated by $H$ is $\tau$-dense in~$X^*$.
\item If $C_1, C_2 \in \C$ are closed convex such that there exists $f \in X^*$ with $\sup f(C_1) < \inf f(C_2)$, then there exist disjoint closed balls $B_1, B_2$ such that $C_i \subseteq B_i$, $i = 1, 2$.
\item $X$ has the $\C$-MIP.
\item Every $f \in S(X^*)$ is a $\C$-semidenting point of $B(X^*)$.
\item For every norm dense subset $A \subseteq S(X)$ and every support mapping $\phi$, the cone generated by $\phi (A)$ is $\tau$-dense in $X^*$.
\el
Then $(a)\implies (b) \implies (c) \implies (d) \iff (e) \implies (f)$. Moreover, if
\[
A = \{x\in S(X) : D(x) \mbox{ contains a $\C$-denting point of } B(X^*)\}
\]
is norm dense in $S(X)$, then all the statements are equivalent.
\end{theorem}

\begin{remark} \rm
In \cite[Theorem 1.10]{CL1}, it is shown that $(a)$ and $(d)$ are equivalent under a weaker assumption that every w*-slice of $B(X^*)$ contains a $\C$-denting point.
\end{remark}

\section{Strong $\C$-MIP}

\begin{definition} \rm
A Banach space $X$ is said to have the strong $\C$-MIP if for every closed convex $C \in \C$ and $\beta \geq 0$, $\ol{C + \beta B(X)}$ is the intersection of closed balls containing it.
\end{definition}

Vanderwerff \cite{Va} called this $\K$-IP and $\W$-IP for $\C = \K$ and $\W$, respectively. Clearly, for $\C = \mathcal{B}$, $\C$-MIP and strong $\C$-MIP coincide. See Section~\ref{sec5} for some examples and counterexamples.

\begin{definition} \rm
The duality map on $X$ is said to be norm-$\tau$ quasicontinuous if for every $f \in S(X^*)$, $\e >0$ and $C \in \C$, there exist $\delta > 0$ and $x \in S(X)$ such that $D(S_X \cap B(x, \delta)) \subseteq B_C(f, \e)$.
\end{definition}

In our opinion, the following is a more natural generalisation of both the w*-semidenting points of \cite{CL} and the $\C$-denting points of \cite{CL1} than the $\C$-semidenting points defined in \cite{CD}.

\begin{definition} \rm
We say that $f \in S(X^*)$ is a strong $\C$-semidenting point of $B(X^*)$ if for each $A \in \C$ and $\e > 0$, there exists a w*-slice $S$ of $B(X^*)$ such that $S \subseteq B_{A}(f, \e)$.
\end{definition}

We begin by observing that the strong $\C$-semidenting points can be characterised completely analogous to \cite[Lemma 2.2]{GGS}.

\begin{lemma} \label{lem3}
For $f \in S(X^*)$, the following are equivalent:
\bla
\item $f \in H$, that is, $f \in \bigcap_{\e >0} \ol{D(M_{\e, C}(X))}^{\tau}$ for every $C \in \C$.
\item $f$ is a strong $\C$-semidenting point.
\item For every $\beta \geq 0$ and closed convex $C \in \C$ such that $\inf f(\ol{C +\beta B(X)}) >0$, there exists a closed ball $B[x_0, r_0]$ in $X$ containing $C +\beta B(X)$ with $0 \notin B[x_0, r_0]$.
\item For every $\e >0$ and $C \in \C$, there exist $\del >0$ and $x \in S(X)$ such that $D(S_X \cap B(x, \delta)) \subseteq B_C(f, \e)$.
\el
\end{lemma}

\begin{proof}
$(a) \implies (b)$: For $C \in \C$, let $f \in \bigcap_{\e >0}\ol{D(M_{\e, C}(X))}^{\tau}$.
Let $\e >0$ be given. We have $f \in \ol{D(M_{\e/4, C}(X))}^{\tau}$. So, there exist $\del >0$ and $x \in M_{\e/4, \del, C}(X)$ and $g \in D(x)$ such that $\|f-g\|_C <\e/4$.

Since $x \in M_{\e/4, \del, C}(X)$, we have $d_1(C, x, \del) < \e/4$. By Lemma~\ref{lem1}$(i)$, for $\alpha = \e\del/8$,
\[
d_2(C, x, \alpha) \leq d_1(C, x, \del) + \frac{2\alpha}{\del} <\e/4 + \e/4 =\e/2.
\]
Hence, $diam_C(S(B(X^*), x, \alpha) <\e/2$.

Since $g \in S(B(X^*), x, \del)$, $S(B(X^*), x, \del) \subseteq B_C(f, \e)$.

$(b) \implies (c)$: Let $\beta \geq 0$ and $C \in \C$ such that $\inf f(\ol{C +\beta B(X)}) >0$.

Let $\inf f(\ol{C +\beta B(X)})=\alpha >0$. Now, there exist $\del >0$ and $x \in S(X)$ such that
\[
S := S(B(X^*), x, \del) \subseteq B_C(f, \alpha).
\]

Choose $g \in S$ and $c \in C$. We have,
\[
g(c) = f(c)- \|f-g\|_C > \alpha +\beta -\alpha = \beta.
\]
Hence, for $y \in \ol{C +\beta B(X)}$, $g(y) >0$.

So, by \cite[Theorem 2.6]{BGG}, there is a closed ball $B[x_0, r_0]$ containing $C +\beta B(X)$ such that $0 \notin B[x_0, r_0]$.

$(c) \implies (b)$: Let $C \in \C$ and $\e >0$ be given.

Let $z_0 \in X$ be such that $\|z_0\|=1 + \e/2$ and $f(z_0) >1 +\e/4$.
Let $D = \ol{aco}(C \cup \{z_0\})$. We have $D \in \C$.

Consider $A:=\{x \in D : f(x) \geq \e/4\}$. We have $A \in \C$.

For $\beta = \frac{\e/4}{1 +\e/4}$, $f(\ol{A +\beta B(X)}) \geq \e/4 - \beta >0$.
So, there exists a closed ball $B[x_0, r_0]$ containing $\ol{A +\beta B(X)}$ such that $0 \notin B[x_0, r_0]$.

So, for $r_1 = r_0 -\beta$ and $x_1 = x_0$,
$A \subseteq B[x_1, r_1]$ and $d(0, B[x_1, r_1]) >\beta$.

Consider $S = \{g \in B(X^*) : g(x_1/\|x_1\|) >(r_1 + \beta)/\|x_1\|\}$.

For $g \in S$, $g(x_1/\|x_1\|) >(r_1 + \beta)/\|x_1\|$. So, $g(x_1) >r_1 + \beta$ which implies $\inf g(B[x_1, r_1]) \geq \beta$.
Thus, $g(x) >0$ for all $x \in A$.
Therefore, applying Phelps' Lemma \cite[Theorem 2.5]{BGG} to the semi-norm $\|\cdot\|_D$,
\[
\left\|f - \frac{\|f\|_D}{\|g\|_D}g\right\|_D <\e/2.
\]

Consider $\ds y_0 = \frac{\e}{4} \frac{z_0}{f(z_0)}$. Since $f(z_0) > \e/4$, $D$ is absolutely convex, and $f(y_0) = \e/4$, $y_0 \in A$, and hence, $g(y_0) > \beta$.

Therefore, $\ds g(z_0) = \frac{f(z_0) g(y_0)}{\e/4} > \frac{(1+\e/4) \beta}{\e/4} = 1$ and so, $\|g\|_D \geq 1$.

Since $f, g \in B(X^*)$ and $D \subseteq (1+ \e/2) B(X)$,
\[
1 \leq \|f\|_D, \|g\|_D \leq 1 + \e/2.
\]
Hence,
\[
\big|\|f\|_D - \|g\|_D\big| <\e/2.
\]
Finally,
\[
\|f - g\|_C \leq\|f - g\|_D \leq \left\|f - \frac{\|f\|_D}{\|g\|_D}g\right\|_D + \big|\|f\|_D - \|g\|_D\big| <\e.
\]

$(b) \implies (d)$: Let $f \in S(X^*)$ be a strong $\C$-semidenting point.
Let $\e >0$ and $C \in \C$ be given.
There exist $\del >0$ and $x \in S(X)$ such that
\[
S(B(X^*), x, \del) \subseteq B_C(f, \e/4).
\]
So, $d_2(C, x, \del) = diam_{C}(S(B(X^*), x, \del)) <\e/2$ and $D(x) \subseteq B_C(f, \e/4)$.
Now, by Lemma ~\ref{lem1}, $d_3(C, x, \del) = diam_{C}(D(S(X)\cap B (x, \delta))) <\e/2$.
Hence,
\[
D(S(X)\cap B (x, \delta)) \subseteq B_C(f, \e).
\]

$(d) \implies (a)$: To show $f \in \ol{D(M_{\e, C}(X))}^{\tau}$ for every $C \in \C$ and $\e > 0$, let $0 <\eta <\e$ and $C_1 \in \C$. Enough to show that there exist $x \in M_{\e, C}$ and $g \in D(x)$ such that $\|f-g\|_{C_1} < \eta$.

Let $C_0 = C \cup C_1 \in \C$. By $(d)$, there exist $\del >0$ and $x \in S(X)$ such that $D(S(X) \cap B(x, \delta)) \subseteq B_{C_0}(f, \eta/2)$.

So, $d_3(C_0, x, \del) \leq \eta$ and $D(x) \subseteq B_{C_0}(f, \eta/2)$. By Lemma ~\ref{lem1}, $d_1(C_0, x, \del/2) \leq d_3(C_0, x, \del) \leq \eta$. Hence, $x \in M_{\eta, \del/2, C_0} \subseteq M_{\e, \del/2, C} \subseteq M_{\e, C}$ and $\|f-g\|_{C_1} < \eta$ for any $g \in D(x)$.
\end{proof}

Here is our main theorem characterising the strong $\C$-MIP. Again, we obtain a complete analogue of \cite[Theorem 2.1]{GGS}.

\begin{theorem} \label{thm1}
For a Banach space $X$, the following are equivalent:
\bla
\item $X$ has the strong $\C$-MIP.
\item Every $f \in S(X^*)$ is a strong $\C$-semidenting point.
\item For every $\e >0$ and $C \in \C$, $D(M_{\e, C}(X))$ is $\tau$-dense in $S(X^*)$.
\item The duality map on $X$ is norm-$\tau$ quasicontinuous.
\item Every support mapping in $X$ maps norm dense sets in $S(X)$ to $\tau$-dense sets in $S(X^*)$.
\el
\end{theorem}

\begin{proof}
Equivalence of $(a)$, $(b)$, $(c)$ and $(d)$ follows from Lemma \ref{lem3}.

$(d) \implies (e)$: Let $A \subseteq S(X)$ be norm dense in $S(X)$. Let $\phi : S(X) \to S(X^*)$ be a support mapping. Let $f \in S(X^*)$, $C \in \C$ and $\e >0$. It is enough to find $x_0 \in A$ such that $\|f - \phi(x_0)\|_C < \e$.

By $(d)$, there exist $\delta > 0$ and $x \in S(X)$ such that $D(S_X \cap B(x, \delta)) \subseteq B_C(f, \e)$. Since, $A$ is dense in $S(X)$, there exists $x_0 \in A$ such that $\|x_0 - x\| <\del$. If follows that $D(x_0) \subseteq B_C(f, \e)$. Therefore, for any support mapping $\phi$, $\|f - \phi(x_0)\|_C < \e$.

$(e) \implies (d)$: Suppose the duality map is not norm-$\tau$ quasicontinuous. Then there exist $f \in S(X^*)$, $\e >0$ and $C \in \C$, such that for any $x \in S(X)$ and $n \in \mathbb{N}$, there exists $y_{n, x} \in S(X)$ and $g_{n, x} \in D(y_{n, x})$ such that $\|y_{n, x} - x\| < 1/n$ and $\|g_{n, x} - f\|_C \geq \e$.

Let $\phi : S(X) \to S(X^*)$ be a support mapping that maps $y_{n, x}$ to $g_{n, x}$. Let $A = \{y_{n, x} : x \in S(X), n\in \mathbb{N}\}$. Then $A$ is norm dense in $S(X)$, but $\phi(A)$ is not $\tau$-dense in $S(X^*)$.
\end{proof}

\begin{remark} \rm
For $\C = \mathcal{K}$ or $\mathcal{W}$, Vanderwerff \cite{Va} observed that the strong $\C$-MIP is equivalent to the following condition:
\bla
\item[$(b')$]
For every $\e >0$ and $C \in \C$, the points in $S(X^*)$ that lie in w*-slices of $B(X^*)$ having $C$-diameter $\leq \e$ is $\tau$-dense in $S(X^*)$.
\el

It is easy to see that this condition is equivalent to $(b)$ above and holds for any compatible family $\C$.
Now consider the following condition:
\bla
\item[$(b'')$]
The $\C$-denting points of $B(X^*)$ are $\tau$-dense in $S(X^*)$.
\el
Clearly, $(b'') \implies (b)$. Moreover, if every w*-slice of $B(X^*)$ contains a $\C$-denting point, then both are equivalent. Notice that this condition holds in any Asplund space.
\end{remark}

Since the $\tau_{\mathcal{F}}$ and $\tau_{\K}$ topologies coincide on bounded sets, we obtain
\begin{corollary}
For a Banach space $X$, the following are equivalent:
\bla
\item $X$ has the strong $\mathcal{K}$-MIP.
\item $X$ has the strong $\mathcal{F}$-MIP.
\item The extreme points of $B(X^*)$ are w*-dense in $S(X^*)$.
\el
\end{corollary}

Though we have given enough evidence that the strong $\C$-semidenting points are more natural than the $\C$-semidenting points, one question remains unanswered: Can we characterise $\C$-MIP in terms of strong $\C$-semidenting points?

So let us consider the condition: \bla
\item[$(a')$]The cone generated by strong $\C$-semidenting points of $B(X^*)$ is $\tau$-dense in~$X^*$.
  \el
It is clear from Theorem~\ref{T1} that this condition is generally stronger than the $\C$-MIP and is equivalent to it if every w*-slice of $B(X^*)$ contains a strong $\C$-semidenting point, a condition slightly weaker than the one that gives equivalence in Theorem~\ref{T1}.

It is well known that a Banach space with a Fr\'echet smooth norm has the MIP. Indeed, this was known to Mazur himself. Notice that a compatible class $\C$ is a bornology as defined in \cite[pg. 64]{Ph}. Using the definition of $\C$-smoothness discussed in \cite[pg. 64]{Ph}, we have the following corollary.
\begin{corollary}
If the norm on $X$ is $\C$-smooth, then $X$ has the strong $\C$-MIP.
\end{corollary}

\begin{remark} \rm
For $\C = \mathcal{K}$ and $\mathcal{W}$, this was observed by Vanderwerff in \cite{Va}.
\end{remark}
\pagebreak

\section{$\C$-UMIP and Strong $\C$-UMIP}

In this section, we discuss uniform versions of $\C$-MIP and strong $\C$-MIP.

The following quantitative result is contained in the proof of \cite[Theorem 2.6]{BGG}. We will need it repeatedly.

\begin{lemma} \label{lem2}
Let $A \subseteq X$ be such that $M = \sup \{\|x\| : x \in A\}$ and $d = d(0, A) > 0$. Suppose there exist $x_0 \in S(X)$ and $0 < \gamma < 1$ such that
\[
S = \{f \in B(X^*) : f(x_0) > \gamma\} \subseteq \{f \in B(X^*) : f(x) > 0 \mbox{ for all } x\in A\}.
\]
Then for $\lambda \geq M/(1-\gamma)$,
\[A \subseteq B\left[\lambda x_0, \lambda - \frac{d(1-\gamma)}{1+\gamma}\right].
\]
\end{lemma}

We observe that there are three ways of defining uniform versions of the strong $\C$-MIP ($(a)$, $(b)$ and $(c)$ below) and they are equivalent.

\begin{theorem} \label{thm2}
In a Banach space $X$, the following are equivalent:
\bla
\item For every $\e >0$, $\beta \geq 0$, $M \geq 2$ and $0 <\eta < \e$, there exists $K >0$ such that for every closed convex $C \in \mathcal{C}$ with $diam (C) \leq M$ and $d(0, C +\beta B(X)) \geq \e$, there exist $z_0 \in X$, $0< r \leq K$ such that $C +\beta B(X) \subseteq B[z_0, r]$ and $d(0, B[z_0, r]) > \eta$.

\item For every $\e >0$, $M \geq 2$ and $0 <\eta < \e$, there exists $K >0$ such that for every closed convex $C \in \mathcal{C}$ with $diam (C) \leq M$ and $d(0, C) \geq \e$, there exist $z_0 \in X$, $0< r \leq K$ such that $C \subseteq B[z_0, r]$ and $d(0, B[z_0, r]) > \eta$.

\item For every $\e >0$, $\beta \geq 0$ and $M \geq 2$, there exist $K >0$ and $0 <\eta < \e$, such that for every closed convex $C \in \mathcal{C}$ with $diam (C) \leq M$ and $d(0, C +\beta B(X)) \geq \e$, there exist $z_0 \in X$, $0< r \leq K$ such that $C +\beta B(X) \subseteq B[z_0, r]$ and $d(0, B[z_0, r]) > \eta$.

\item For every $\e >0$, there exists $\del >0$ such that for every $C \in \C$, $C \subseteq B(X)$ and $f \in S(X^*)$, there is $x \in S(X)$ such that
\[
S(B(X^*), x, \del) \subseteq B_C(f, \e).
\]

\item For every $\e >0$, there exists $\del >0$ such that for every $C \in \C$, $C \subseteq B(X)$ and $f \in S(X^*)$, there is $x \in S(X)$ such that $D(S_X \cap B(x, \delta)) \subseteq B_C(f, \e)$.

\item For every $\e >0$, there exists $\del >0$ such that for every $C \in \C$, $C \subseteq B(X)$ and $f \in S(X^*)$, there exists $x \in M_{\e, \del, C}$ such that $D(x) \subseteq B_C(f, \e)$.
\el
\end{theorem}

\begin{proof}
Notice that for any bounded $C \subseteq X$ and $r, \beta > 0$,
\[
C \subseteq B[z_0, r] \iff C + \beta B(X) \subseteq B[z_0, r+\beta]
\]
and
\[
d(0, B[z_0, r + \beta]) = d(0, B[z_0, r]) - \beta
\]
provided $d(0, B[z_0, r]) > \beta$.

Clearly, $(a) \implies (b)$ and $(a) \implies (c)$.

$(b) \implies (d)$: This is a quantitatively more precise version of the proof of Lemma~\ref{lem3} ($(c) \implies (b)$).

Let $\e >0$ be given. Let $\alpha = \frac{\e}{4+\e} > 0$ and $\eta \in (\alpha, \e/4)$.  Choose $K$ for $\e/4$, $\eta$ and $M = 4$. That is, for every $C \in \C$ with $d(0,C) \geq \e/4$ and $diam (C) \leq 4$, there exists a closed ball $B[z_0,r]$ containing $C$ with $r \leq K$ and $d(0, B[z_0,r]) >\eta$.
We will show that $\delta = \frac{\eta-\alpha}{K + \eta}$ works.

Let $C \in \C$ such that $C \subseteq B(X)$ and $f \in S(X^*)$.
As in the proof of Lemma~\ref{lem3} ($(c) \implies (b)$), let $z_0 \in X$ be such that $\|z_0\|=1 + \e/2$ and $f(z_0) >1 +\e/4$ and $D = \ol{aco}(C \cup \{z_0\})$. Again, $D \in \C$ and $\|f\|_D \geq 1$. For $A = \{y \in D: f(y) \geq\e/4\}$, $A \in \C$, diam $(A) \leq 4$ and $d(0, A) \geq \e/4$.

By assumption, there exist $x_0 \in X$ and $r > 0$ such that $A \subseteq B[x_0, r]$, $d(0, B[x_0, r]) > \eta$ and $r \leq K$.
Thus, $\|x_0\| >r + \eta$.
Now let
\[
S := S\left(B(X^*), \frac{x_0}{\|x_0\|}, \del\right) = \left\{g \in B(X^*): g\left(\frac{x_0}{\|x_0\|}\right) > \frac{K+\alpha} {K + \eta}\right\}.
\]

Then, for $g \in S$, we have $g(x_0/\|x_0\|) > (K + \alpha)/(K + \eta) >(r + \alpha)/(r + \eta)$. So, $g(x_0) >(r + \alpha)\|x_0\|/(r + \eta)$. But, $\|x_0\| >r+\eta$. Thus,
$g(x_0) > r + \alpha$. So, $\inf g(B[x_0, r]) > g(x_0) -r > \alpha > 0$. And hence, again as in Lemma~\ref{lem3} ($(c) \implies (b)$), applying Phelps' Lemma \cite[Theorem 2.5]{BGG} to $\|.\|_D$,
\[
\left\|f - \frac{\|f\|_D}{\|g\|_D} g\right\|_D < \e/2.
\]

Now, arguing as in Lemma~\ref{lem3} ($(c) \implies (b)$), we have that for $g \in S$,
\[
\|f - g\|_C \leq\|f - g\|_D < \e.
\]

$(c) \implies (d)$: This is a minor variant of the above proof.

Let $\e >0$ be given. Let $\alpha = \frac{\e}{4+\e} > 0$ and $\beta \in (\alpha, \e/4)$. Choose $K$ and $\eta$ for $\e/4 - \beta$, $\beta$ and $M = 4$. That is, for every $C \in \C$ with $d(0, C +\beta B(X)) \geq \e/4-\beta$ and $diam (C) \leq 4$, there exists a closed ball $B[z_0, r]$ containing $C + \beta B(X)$, $d(0, B[z_0, r]) > \eta >0$ and $r \leq K$.

We will show that $\delta = \frac{\beta-\alpha}{K + \beta}$ works.

Let $C \in \C$ such that $C \subseteq B(X)$ and $f \in S(X^*)$.

As before, let $z_0 \in X$ be such that $\|z_0\|=1 + \e/2$ and $f(z_0) >1 +\e/4$. Let $D = \ol{aco}(C \cup \{z_0\})$ and $A = \{y \in D: f(y) \geq\e/4\}$. Then $D, A \in \C$, $\|f\|_D \geq 1$, diam$(A) \leq 4$ and $d(0, A + \beta B(X)) \geq \e/4 -\beta$.

By assumption, there exist $x_0 \in X$ and $0 < r \leq K$ such that $A +\beta B(X) \subseteq B[x_0, r]$, and $d(0, B[x_0, r]) > \eta$. Thus, $\|x_0\| >r + \eta$ and $A \subseteq B[x_0, r-\beta]$, $d(0, B[x_0, r-\beta]) \geq \eta +\beta$ and $r-\beta \leq K$.

Rest of the proof is identical as above and shows that
\[
S\left(B(X^*), \frac{x_0}{\|x_0\|}, \del\right)
\subseteq B_C(f, \e).
\]

$(d) \implies (a)$: Let $\e >0$, $\beta \geq 0$, $M \geq 2$ and $0 <\eta < \e$ be given.

Let $L = M + \e +2\beta$. Clearly, $L > 1$.

Let $\e_1 = \e/L$ and $\eta_1 = \eta/L$. Choose $0< \del <1$ such that for any $C_1 \in \mathcal{C}$ with $C_1 \subseteq B(X)$ and $f \in S(X^*)$, there exists $x \in S(X)$ such that
\[
S(B(X^*), x, \del) \subseteq B_{C_1}(f, \e_1-\eta_1).
\]

Let $K = (L+\e)/\del$.

Let $C \in \C$ with $diam(C) \leq M$ and $d(0, C +\beta B(X)) \geq \e$. We will show that there is a closed ball $B$ of radius $\leq K$ such that $C +\beta B(X) \subseteq B$ and $d(0, B) >\eta$.

Case I: $C + \beta B(X) \setminus B[0, L] \neq \emptyset$.

If $z \in C + \beta B(X) \setminus B[0, M + \e + 2\beta]$, then
$d(0, B[z, M + 2 \beta]) \geq \e > \eta$, $C + \beta B(X) \subseteq B[z, M + 2 \beta]$ and $M + 2 \beta \leq K$.

Case II: $C +\beta B(X) \subseteq B[0, L]$.

Let $\widetilde{C} = \frac{1}{L}C$ and $\alpha = \beta/L$. Then, $\widetilde{C} \subseteq D := \frac{1}{L}(C + \beta B(X)) = \widetilde{C} + \alpha B(X) \subseteq B(X)$ and $d(0, D) \geq \e_1$. Thus, $d(0, \widetilde{C}) \geq \e_1 + \alpha$. It suffices to show that there is a closed ball $B$ of radius $\leq K/L$ such that $D \subseteq B$ and $d(0, B) > \eta_1$. Let $D' = D + \eta_1 B(X) = \widetilde{C} + (\eta_1 + \alpha) B(X)$.

Choose $f \in S(X^*)$ such that $\inf f(\widetilde{C}) \geq \e_1 + \alpha$. For this $f$, there exists $x \in S(X)$ such that
\[
S :=S(B(X^*), x, \del) \subseteq B_{\widetilde{C}}(f, \e_1-\eta_1).
\]

Let $h \in S$.
For $y \in \widetilde{C}$, we have,
\[
h(y) \geq f(y) - \|f - h\|_{\widetilde{C}} > (\e_1 + \alpha) - (\e_1-\eta_1) = \eta_1 + \alpha.
\]
That is, $h(y) >0$ for all $y \in D'$. So,
\[
S \subseteq \{g \in X^* : g(x) > 0 \text{ for all } x \in D'\}.
\]

By Lemma~\ref{lem2}, we have that for $\lambda = \frac{(1 +  \e_1)}{\del}$
\[
D' \subseteq B\left[\lambda x, \lambda - \frac{d(0, D')\del}{2-\del}\right].
\]
Hence,
\[
D \subseteq B\left[\lambda x, \lambda - \frac{d(0, D')\del}{2-\del}  - \eta_1\right].
\]

Also,
\[
d(0, B\left[\lambda x, \lambda - \frac{d(0, D')\del} {2-\del} - \eta_1\right]) \geq \frac{d(0, D')\del}{2-\del} + \eta_1 > \eta_1
\] and
\[
\lambda - \frac{d(0, D')\del}{2-\del} -\eta_1 \leq \lambda = \frac{1 + \e_1}{\del} = K/L.
\]

This completes the proof.

Equivalence of $(d)$, $(e)$ and $(f)$ follows from Lemma~\ref{lem1}.
\end{proof}

\begin{definition} \rm
For a compatible class $\C$, $X$ is said to have strong the $\C$-UMIP if any of the equivalent conditions of Theorem~\ref{thm2} is satisfied.
\end{definition}

We again observe that

\begin{corollary}
Strong $\mathcal{K}$-UMIP and strong $\mathcal{F}$-UMIP are equivalent.
\end{corollary}

Now we come to a uniform version of the $\C$-MIP.

\begin{definition} \rm
For a compatible class $\C$, $X$ is said to have the $\C$-UMIP if for every $\e >0$ and $M \geq 2$, there exist $K >0$ and $0 <\eta \leq \e$ such that for every closed convex $C \in \C$ with $diam(C) \leq M$ and $d(0, C) \geq \e$, there exist $z_0 \in X$, $0< r \leq K$ such that $C \subseteq B[z_0, r]$ and $d(0, B[z_0, r]) \geq \eta$.
\end{definition}

It can be similarly characterised in terms of the uniform version of the semidenting point discussed in \cite{CD}. The proof is a slight modification of the arguments of Theorem~\ref{thm2}. We omit the details.

\begin{theorem} \label{thm3}
For a compatible class $\C$, the following are equivalent:
\bla
\item $X$ has the $\C$-UMIP
\item For every $\e >0$, there exists $\del >0$ such that for every $C \in \C$, $C \subseteq B(X)$ and $f \in S(X^*)$, there is $x \in S(X)$ such that
\[
S(B(X^*), x, \del) \subseteq cone\ (B_{C}(f, \e)).
\]
\el
\end{theorem}

\section{Summary} \label{sec5}
In this section, we summarise the interrelations between various notions discussed above. We give counterexamples to some of the reverse implications. As far as we know, the other implications are open.

\[
\begin{array}{ccccccc}
\mbox{MIP} & \overset{\Rightarrow}{\not \Leftarrow} & \mbox{strong $\mathcal{W}$-MIP} & \Rightarrow & \mbox{strong $\mathcal{K}$-MIP} & \Leftrightarrow &\mbox{strong $\mathcal{F}$-MIP} \\
\Updownarrow & & \Downarrow & & \Downarrow \not \Uparrow & & \Downarrow \\
\mbox{MIP} & \overset{\Rightarrow}{\not \Leftarrow} & \mbox{$\mathcal{W}$-MIP} & \Rightarrow & \mbox{$\mathcal{K}$-MIP} & \Rightarrow &\mbox{$\mathcal{F}$-MIP} \\
\end{array}
\]

\subsection*{Examples}
\begin{itemize}
\item Vanderwerff in \cite[Theorem 2.6]{Va} shows that there is a Banach space which has $\mathcal{W}$-MIP but it does not have the strong $\mathcal{K}$-MIP. Thus, the $\mathcal{K}$-MIP does not imply the strong $\mathcal{K}$-MIP. This also shows that there exist Banach spaces with the $\mathcal{W}$-MIP but not the MIP.
\item Vanderwerff in \cite[Theorem 1.1]{Va} also shows that every Banach space can be renormed to have the $\mathcal{W}$-MIP, and hence, the $\mathcal{K}$-MIP. This again shows that the $\mathcal{W}$-MIP does not imply the MIP.
\item In $\ell_1$, G\^ateaux differentiability and weak Hadamard (that is, $\mathcal{W}$-) differentiability coincide (see \cite{GS}). Now, $\ell_1$, being a separable space, has an equivalent norm that is G\^ateaux smooth, and hence, $\mathcal{W}$-smooth, and therefore, has the strong $\mathcal{W}$-MIP. But, being a separable non-Asplund space, it has no MIP renorming.
\end{itemize}

We also have similar interrelations between the uniform versions. In this case, most of the reverse implications are open.
\[
\begin{array}{ccccccc}
\mbox{UMIP} & \Rightarrow & \mbox{strong $\mathcal{W}$-UMIP} & \Rightarrow & \mbox{strong $\mathcal{K}$-UMIP} & \Leftrightarrow &\mbox{strong $\mathcal{F}$-UMIP} \\
\Updownarrow & & \Downarrow & & \Downarrow & & \Downarrow \\
\mbox{UMIP} & \Rightarrow & \mbox{$\mathcal{W}$-UMIP} & \Rightarrow & \mbox{$\mathcal{K}$-UMIP} & \Rightarrow &\mbox{$\mathcal{F}$-UMIP}
\end{array}
\]

\bibliographystyle{amsplain}

\begin{thebibliography}{10}
\bibitem{Ba1} Pradipta Bandyopadhyay, \emph{The Mazur Intersection Property in Banach Spaces and Related Topics}, Ph. D Thesis submitted to ISI, Calcutta, 1991.

\bibitem{Ba2} Pradipta Bandyopadhyay, \emph{The Mazur Intersection Property for Families of Closed Bounded Convex Sets in Banach Spaces}, Colloq. Math. LXIII (1992), 45--56.

\bibitem{BGG} Pradipta Bandyopadhyay, Jadav Ganesh and Deepak Gothwal, \emph{On Uniform Mazur Intersection Property} Studia Math. 260 (2021), no. 3, 273--283.

\bibitem{CL} D. Chen and B. -L. Lin, \emph{On B-Convex and Mazur Sets of Banach Spaces}, Bull. Pol. Acad. Sci. Math., 43 (1995), 191-198.

\bibitem{CL1} D. Chen, B. -L. Lin, \emph{Ball Separation Properties in Banach Spaces}, Rocky Mountain J. Math. 28 (1998), 835-873

\bibitem{CD} Q. Cheng, Y. Dong, \emph{Characterization of Normed Linear Spaces with Generalized Mazur Intersection Property}, Studia Math. 219 (2013), no. 3, 193--200.

\bibitem{GGS} J. R. Giles, D. A. Gregory, B. Sims, \emph{Characterisation of normed linear spaces with Mazur's intersection property}, Bull. Austral. Math. Soc., \textbf{18} (1978), no. 1, 105-123.

\bibitem{GS}J. R. Giles, S. Sciffer, \emph{On weak Hadamard differentiability of convex functions on Banach spaces}, Bull. Austral. Math. Soc. 54 (1996), no. 1, 155–166.

\bibitem{GJM} A. S. Granero, M. Jim\'enez-Sevilla, J. P. Moreno, \emph{Intersections of closed balls and geometry of Banach spaces}, Extracta Math., 19 (2004), 55-92.

\bibitem{Ma} S.\ Mazur, \emph{\"Uber Schwache Konvergenz in den Ra\"umen $(L^p)$}, Studia Math.\ {\bf 4} (1933), 128--133.

\bibitem{Ph} R. R. Phelps, \emph{Convex functions, monotone operators and differentiability}, Lecture Notes in Mathematics, 1364, Springer-Verlag, Berlin, 1989.

\bibitem{Se} A.\ Sersouri, {\em The Mazur Property for Compact Sets}, Pacific J.\ Math.\ {\bf 133} (1988), 185--195.

\bibitem{S2} A.\ Sersouri, {\em Mazur's Intersection Property for Finite Dimensional Sets}, Math.\ Ann.\ {\bf 283} (1989), 165--170.

\bibitem{Va} J. Vanderwerff, \emph{Mazur Intersection Properties for Compact and Weakly Compact Convex Sets}, Canad. Math. Bull. 41 (1998), no. 2, 225--230.

\bibitem{WZ} J. H. M. Whitfield, V. Zizler, \emph{Mazur's intersection property of balls for compact convex sets}, Bull. Austral. Math. Soc. 35 (1987), no. 2, 267--274.

\bibitem{WZ1} J. H. M. Whitfield, V. Zizler, \emph{Uniform Mazur's Intersection Property of Balls}, Canad. Math. Bull., 30 (1987), 455-460.

\bibitem{Z} V.\ Zizler, {\em Renorming Concerning Mazur's Intersection Property of Balls for Weakly Compact Convex Sets}, Math.\ Ann.\ {\bf 276} (1986), 61--66.
\end{thebibliography}

\end{document}